\documentclass[twoside,12pt]{article}
\usepackage{amssymb,amsfonts,amsmath,amsthm}
\usepackage{graphics,graphicx}

\newtheorem{theorem}{Theorem}
\newtheorem{lemma}{Lemma}
\newtheorem{remark}{Remark}

\def\RR{$\lim\limits_{i\to \infty }{\frac{V_{i+1}}{V_{i}}}$}
\def\RS{$\lim\limits_{i\to \infty }{\frac{V_{i}}{S_{i}}}$}
\def\RrS{$\lim\limits_{i\to \infty }{\frac{V_{i+1}}{S_{i}}}$}

\def\SS{$\lim\limits_{i\to \infty }{\frac{S_{i+1}}{S_{i}}}$}

\def\bb{$\lim\limits_{i\to \infty }{\frac{b_{i+1}^d}{b_{i}^d}}$}

\def\balpha{\text{\boldmath $\alpha$}}

\newcommand*{\ADR}{University of Sopron,  Institute of Mathematics, Hungary. \textit{nemeth.laszlo@uni-sopron.hu}}
\newcommand*{\TIT}{The growing ratios of hyperbolic regular mosaics with bounded cells}
\title{\bf \TIT}

\author{L\'aszl\'o N\'emeth\footnote{\ADR}}
\date{}

\begin{document}

\title{\bf The growing ratios of hyperbolic regular mosaics with bounded cells}

\maketitle

\begin{abstract}
    In 3- and 4-dimensional hyperbolic spaces there are four, respectively five, regular mosaics with bounded cells. A belt can be created around an arbitrary base vertex of a mosaic. The construction can be iterated and a growing ratio can be determined by using the number of the cells of the considered belts. In this article we determine these growing ratios for each mosaic in a generalized way. \\[1mm]
{\em Key Words: hyperbolic mosaic, regular mosaic, regular tessalation, honeycomb.}\\
{\em MSC code:  52C22, 52B11, 05B45.} 

    The final publication is available at Armenian Journal of Mathematics via \\http://www.flib.sci.am/eng/journal/Math/. 
\end{abstract}

\section*{Introduction}\label{introduction}

{\sc Fejes T\'oth, L.} \cite[p.261.]{fej} examined the area of circles with common centre in the
following way. Let $C(r)$ be the area enclosed by the circle of radius $r$. If $a>0$, then $\lim\limits_{r\to \infty }
{\frac{C(r+a)-C(r)}{C(r)}}$ is equal to $0$ in the Euclidean plane but it is $e^a-1$ in the hyperbolic plane.
This fact inspired several mathematicians to deal with the structure of hyperbolic space (\cite{hor}, \cite{kar}, \cite{nem2}, \cite{ver2}, \cite{zei}).

%$\lim\limits_{r\to \infty }\left( \left( {C(r+a)-C(r)}\right) /{C(r)}\right) $

In the following we examine two similar limits for $d$-dimensional hyperbolic regular mosaics
${\cal H}$ (regular tessellations or honeycombs \cite{cox}). In $d$-dimensional spaces  the Schl\"afli symbol of a regular mosaic is
$\{n_1,n_2,\dots ,n_d\}$, where $\{n_1,n_2,\dots ,n_{d-1}\}$ is the cell ${\cal P}$
({\sc Coxeter} \cite{cox}).

Let us fix a cell ${\cal P}$ (or a vertex) as the belt 0. Denote it by $B_0$ and create belts around it. The
first belt consists of the cells of the mosaic having common (finite) points with $B_0$. They are neighbours of ${\cal P}$.
If the belt $i$ is known, let the belt $(i+1)$ consist of the cells that have a common (finite) point
(not necessarily a common vertex) with the belt $i$, but have no common point with the belt $(i-1)$.

Let $V_i$ denote the volume of the belt $i$ and $S_i=\sum_{j=0}^iV_j$. 
We call the limit \RR\  the \emph{growing ratio} or \emph{crystal-growing ratio}. The definition was suggested by {\sc I. Vermes} \cite{ver2}. 
The often examined limit \RS\ also gives information about the growing of regular mosaics. 

{\sc K\'arteszi} \cite{kar} examined the mosaics with regular triangles $\{3,m\}$ in the hyperbolic plane. He
took a triangle as the belt~$0$ and constructed the other belts around it. He calculated that \RS\
$=\frac{\sqrt{(m-4)^2-4}-(m-6)}{2}$,  $(m>6)$.  {\sc Horv\'ath} \cite{hor} showed for all regular mosaics
$\{p,q\}$ in the hyperbolic plane, that  \RS\ $=\frac{\sqrt{c^2-4}-(c-2)}{2}$, where $c>2$ and
$c=(p-2)(q-2)-2$. {\sc Vermes}  \cite{ver2} gave this limit for mosaics with asymptotic
polygons in the hyperbolic plane. {\sc Zeitler} \cite{zei} determined for the mosaic $\{4,3,5\}$ in 3-dimensional hyperbolic space, that  \RS\ $=4\sqrt{14}-14$ and \RR\
$=15+4\sqrt{14}$. The author \cite{nem2} determined, that the \RR\,$=$\SS\,$\approx2381.8$ and
\RS\,$\approx0.9996$ in case of 4-dimensional mosaic $\{4,3,3,5\}$ and $\{5,3,3,4\}$ and calculated the limits for eight further
mosaics with unbounded cells  (\cite{nem1}).

\textsc{Coxeter} \cite{cox} examined the higher dimensional hyperbolic regular mosaics  too. He proved, that in 3-dimensional hyperbolic space the  regular mosaics with bounded cells
are the mosaics $\{3,5,3\}$, $\{4,3,5\}$, $\{5,3,4\}$ and $\{5,3,5\}$, all the others have 
unbounded cells (see \cite{Martelli}). In 4-dimensional hyperbolic space the mosaics 
$\{3,3,3,5\}$, $\{4,3,3,5\}$,
$\{5,3,3,5\}$, $\{5,3,3,4\}$, $\{5,3,3,3\}$ have  bounded cells. In higher dimensional
hyperbolic spaces there is no regular mosaic  with bounded cells.

Let us take a regular $d$-dimensional polyhedron, as a cell of a mosaic. Consider the middle point of the cell, a middle point of a $(d-1)$-dimensional face, a middle point of its $(d-2)$-dimensional face, and so on, finally a vertex of the last edge, and all these points determine a simplex as a characteristic simplex (fundamental domain) of the regular polyhedron. The polyhedron and the regular mosaic are generated by reflections in facets ($(d-1)$-dimensional hyperfaces) of a characteristic simplex.   
(More detailed definition is given in the next section (and in \cite{Martelli,vin})).

In this article we are going to determine the limits defined above in a general way for all the hyperbolic regular mosaics  with bounded cells. Since the examined regular mosaics are bounded,  the volumes of all the cells are finite. Thus, if $r_i$ is the number of the cells in the belt~$i$ and $s_i=\sum_{j=0}^ir_j$, then we can simplify  the limits to  \RR$=\lim\limits_{ i\to \infty }{\frac{r_{i+1}}{r_i}}$ and  \RS$=\lim\limits_{ i\to \infty }{\frac{r_{i}}{s_i}}$.
The method consists of three main steps using characteristic simplices  of the cells. 
First, we take the characteristic simplices  of the mosaics and we construct  belts  with characteristic simplices around their vertices and we examine these simplex-belts.
Second, we form belts with cells around the middle points of  $l$-dimensional faces of a cell and we examine the growing. 
Finally, we take the belt $(i+1)$, defined above, as the union of the belts 1 around the vertices of belt~$i$. This yields to a system of linear recurrence sequences with a coefficient matrix $\mathbf{M}$ for each regular mosaic, where the recurrence sequences give the numbers  of the different dimensional vertices of the characteristic simplices in the belts. (See the exact definition later.)
Theorem \ref{tetel:algebrai_pqr} and Table \ref{tablazat:szabalyos_mozaikok veges1} show the final results in 3- and 4-dimensional hyperbolic spaces; of these results, the ratios for the mosaics  $\{3,5,3\}$,  $\{5,3,4\}$, $\{5,3,5\}$,  $\{3,3,3,5\}$,   $\{5,3,3,3\}$ and  $\{5,3,3,5\}$ are new. The other results are of course known (\cite{hor}, \cite{kar}, \cite{nem2}, \cite{zei}), but are found here by a general method.  

\begin{theorem}[Main theorem]\label{tetel:algebrai_pqr}
For the hyperbolic regular mosaics  with bounded cells \\$\{n_1,n_2,\dots ,n_d\}$ 
the growing ratios \RR\ $=z_1$ and limits \RS\ $=\frac{z_1-1}{z_1}$, where $z_1$ is the largest eigenvalue of the coefficient matrix $\mathbf{M}$ of the corresponding system of linear recurrence sequences determined by the regular mosaic. 
\end{theorem}

Now, we give the exact values of the limits for hyperbolic regular mosaics  based on Theorem~\ref{tetel:algebrai_pqr}. The decimal approximations can be seen in Table \ref{tablazat:szabalyos_mozaikok veges1}.

\begin{enumerate}

\item[]\hskip -1cm \emph{\{$p,q$\}}\emph{:} \RR\ $=\frac{pq}{2}-p-q+1+\frac{\sqrt{(q-2)(p-2)(pq-2p-2q)}}{2}=\frac{c+\sqrt{c^2-4}}{2}$ 
and \\ \RS\ $=\frac{\sqrt{c^2-4}-(c-2)}{2}$, where $c=(p-2)(q-2)-2$.
\item[]\hskip -1cm  \emph{\{4,3,5\}} and \emph{\{5,3,4\}}\emph{:} 
         \RR\ $=15+4\sqrt{14}$ and
           \RS\ $=\frac{14+4\sqrt{14}}{15+4\sqrt{14}}={4\sqrt{14}-14}$.
\item[]\hskip -1cm \emph{\{5,3,5\}}\emph{:}  \RR\ $=\frac{167}{2}+\frac{13}{2}\sqrt{165}$ and
            \RS\ $=\frac{165+13\sqrt{165}}{167+13\sqrt{165}}=\frac{13}{2}\sqrt{165}-\frac{165}{2}$.
\item[]\hskip -1cm \emph{\{3,5,3\}}\emph{:}  \RR\ $=\frac{47}{2}+\frac{21}{2}\sqrt{5}$ and
            \RS\ $=\frac{45+21\sqrt{5}}{47+21\sqrt{5}}=\frac{21}{2}\sqrt{5}-\frac{45}{2}$.
           
\item[]\hskip -1cm \emph{\{3,3,3,5\}} and \emph{\{5,3,3,3\}}\emph{:}  \RR\ $=22+\sqrt{401}+c$ and
            \RS\ $=\frac{21+\sqrt{401}+c}{22+\sqrt{401}+c}$  ${=-21-\sqrt{401}+c}$, where $c=2\sqrt{221+11\sqrt{401}}$.        
\item[]\hskip -1cm \emph{\{4,3,3,5\}} and \emph{\{5,3,3,4\}}\emph{:}  \RR\ $=\frac12({1195+11\sqrt{11641}+c})$ and \\
            \RS\ $=\frac12({-1193-11\sqrt{11641}+c})$, where
             $c=\sqrt{2836582+26290\sqrt{11641}}$.
\item[]\hskip -0.9cm \emph{\{5,3,3,5\}}\emph{:}  \RR\ $=79876+3135\sqrt{649}+c$ and                         \RS\ $=-79875-3135\sqrt{649}+c$,  where $c=2\sqrt{3189673350+125205630\sqrt{649}}$.

\end{enumerate}

\begin{table}[!ht]
  \centering
\begin{tabular}{|c|c|c|}
  \hline
    mosaic & \quad \quad \RR \quad \quad & \quad \quad \RS \quad \quad \\
  \hline \hline
%  $\{4,4\}$, $\{3,6\}$, $\{6,3\}$       &   1        &   0      \\ \hline
%  $\{p,q\}$ $(\frac1p+\frac1q<\frac12)$    &   ...        &   ..     \\ \hline
%   \hline \hline
  $\{4,3,5\}$, $\{5,3,4\}$  &   29.9666  &  0.9666  \\ \hline
  $\{3,5,3\}$               &  46.9787  &  0.9787  \\ \hline  $\{5,3,5\}$               &  166.9940  &  0.9940  \\ \hline
 \hline
  $\{3,3,3,5\}$,   $\{5,3,3,3\}$  &  84.0381  &  0.9881     \\ \hline
 $\{4,3,3,5\}$,   $\{5,3,3,4\}$  &  2381.8277  &  0.9996   \\ \hline
 $\{5,3,3,5\}$                &  319483.2496  &  0.999997  \\ \hline
%\hline  $\{4,3,\dots,3,4\}$                 &   1         &  0      \\ \hline
\end{tabular}
\caption{{Limits in the case of hyperbolic mosaics with bounded cells.}\label{tablazat:szabalyos_mozaikok veges1}}
\end{table}

\section{Definitions}\label{definition}

We take a regular mosaic $\{n_1,n_2,\dots ,n_d\}$. An \emph{$x$-point} (or \emph{$x$-vertex}) is  the centre of an $x$-dimensional face ($x$-face) of the regular polyhedron $\{n_1,n_2,\dots,$ $n_{d-1}\}$. 

We consider   0-, 1-, 2-, \dots\  and $d$-points, where  $d$-point is the centre of a cell,  $(d-1)$-point is the centre of its $(d-1)$-face,  $(d-2)$-point is the centre of its $(d-2)$-face, and similarly  $x$-point is the centre of an $x$-face of the $(x+1)$-face ($x=0,1,\ldots d-1$). 

Let a characteristic simplex of a cell be the simplex, defined by the vertices 0-, 1-, 2-, \dots\  and $d$-points, where  $d$-point is the centre of the cell,  $(d-1)$-point is the centre of its $(d-1)$-face, and similarly $x$-point is the centre of an $x$-face of the $(x+1)$-face ($x=0,1,\ldots d-1$) (see \cite{Martelli,vin}).
The reflections in facets ($(d-1)$-dimensional hyperfaces) of a characteristic simplex generate not only cells, but also a regular mosaic. So, it is not only the characteristic simplex of a cell, but also of the mosaic. 
Let us denote a characteristic simplex by $\varDelta$.  
Figure~\ref{abra:hatszog_pq} shows a part of mosaic $\{p,q\}=\{6,3\}$ and a characteristic simplex (triangle). There are around a 2-point $2p$ and around a 0-point $2q$ characteristic triangles. In Figure~\ref{abra:szimplex_pqr} the characteristic simplices of the mosaics $\{p,q,r\}=\{4,3,4\}$ and $\{4,3,5\}$ can be seen. The mosaic $\{4,3,4\}$ is the well-known Euclidean cube mosaic. Around each 23-edge, 30-edge and 01-edge there are $2p$, $2q$, $2r$ characteristic simplices, respectively. Around the other edges there are always 4 simplices.

\begin{figure}[!htb]
		\centering
 \includegraphics{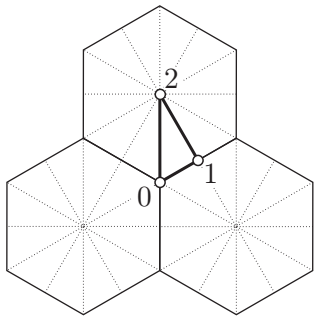} 
  \caption{{The characteristic simplex $\varDelta$ of the mosaic $\{6,3\}$.}}
  \label{abra:hatszog_pq}
\end{figure}

\begin{figure}[!htb]
	\centering
\includegraphics[width=0.95\linewidth]{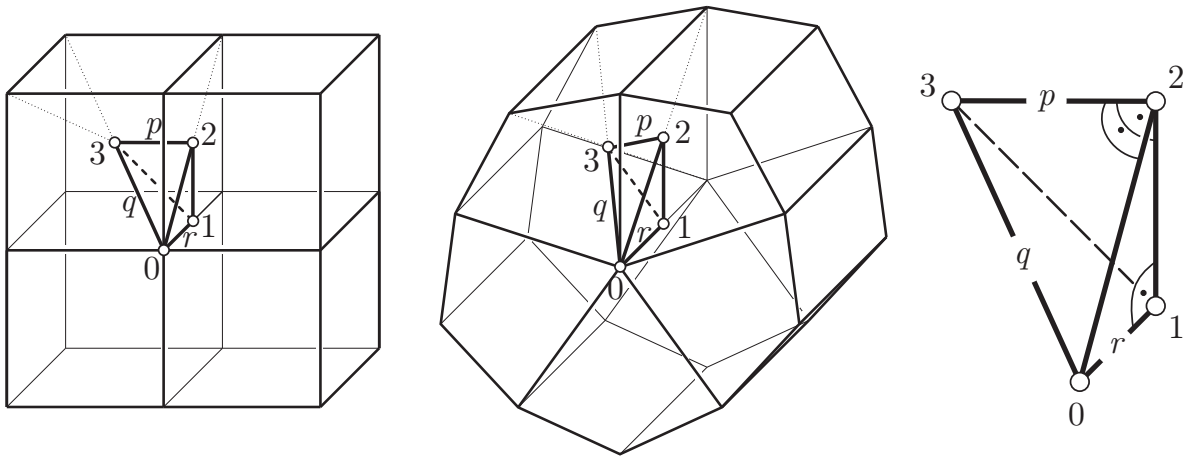}
 \caption{{The characteristic simplices $\varDelta$  of the mosaics $\{p,q,r\}=\{4,3,4\}$ and $\{4,3,5\}$.}} 
 \label{abra:szimplex_pqr}
\end{figure}

We now define several sets needed in what follows. Generally, the capital letters denote sets, the lower-case letters the number of their elements and the bold letters the vectors.

\begin{tabular}{r p{11.0cm}}
$K_x$ :& the set whose elements are the characteristic simplices having a common $x$-point. (Then the union of the elements of $K_d$ is a cell. {Figure~\ref{abra:szimplex_k2}} shows a $K_2$ of the mosaics $\{4,3,4\}$ and  $\{4,3,5\}$).\\
$K_x^y$ :& the set of  $y$-points of $K_x$.\\
$k_x^y$ :& the number of $y$-points of $K_x$ ($k_x^y =\left| K_x^y \right|$, $k_x^x=1$).
\end{tabular}

\begin{tabular}{r p{11.0cm}}
$G_x$ :& the set whose elements are the cells having a common $x$-point. (Figure~\ref{abra:hatszog_pq} shows a $G_0$ and Figure~\ref{abra:szimplex_pqr} shows two $G_1$s. Their cells have a common vertex or common edges, respectively.)\\
$G_x^y$ :& the set of  $y$-points of $G_x$. \\
$g_x^y$ :& the number of  $y$-points of $G_x$ ($ g_x^y = \left| G_x^y \right| $). 
Then $g_x^d = \left| G_x^d \right|$ is the number of the cells having a common $x$-point.
\end{tabular}

\begin{tabular}{r p{11.0cm}}
$B_i$ :& the belt $i$; the set whose elements are the cells in the belt~$i$.\\
$B_i^y$ :& the set of  $y$-points of the belt $i$. \\
$b_i^y$ :& the number of $y$-points of the belt $i$. ($b_i^d$ is the number of the cells of the belt~$i$.)
\end{tabular}

\begin{tabular}{r p{11.0cm}}
$W_i$ :& the union of belts  0, 1, \ldots , $i$  ($W_i= \bigcup\limits_{j=0}^i B_j$).\\
$W_i^y$ :& the set of  $y$-points of $W_i$ ($W_i^y = \bigcup\limits_{j=0}^i B_j^y$).\\
$w_i^y$ :&  the number of  $y$-points of $W_i$  ($w_i^y=|W_i^y|= \sum\limits_{j=0}^i b_j^y$).  \\
$\mathbf{w}_{i}$  :& $\mathbf{w}_{i}:= (w^0_i\ w^1_i\ \dots \ w^d_i)^{T}$.\\
$\mathbf{v}_{i}$ :& $\mathbf{v}_{i}:=(b^0_i\ b^1_i\ \dots \ b^d_i)^{T}$, then $\mathbf{v}_{i+1}=\mathbf{w}_{i+1}-\mathbf{w}_{i}$. For example the number of the cells of the belt $(i+1)$ is $b_{i+1}^d=w_{i+1}^d-w_{i}^d$.
\end{tabular}

\begin{figure}[!htb]
		\centering
 \includegraphics{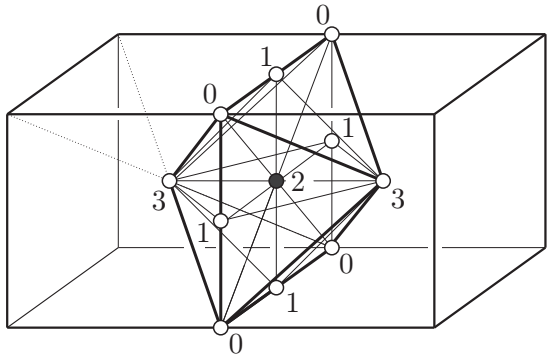} 
\caption{{$K_2$ of the mosaics $\{4,3,4\}$ and $\{4,3,5\}$.}}
\label{abra:szimplex_k2}
\end{figure}

\medskip
Let the belt~$0$ be a cell, then $\mathbf{w}_{0}=\mathbf{v}_{0}=(k_d^0\ k_{d}^{1}\ \ldots \ k_d^{d-1} \ k_d^d)^{T}$  (if the belt $0$ is a vertex, then  $\mathbf{w}_{0}=\mathbf{v}_{0}=(1\ 0 \ldots 0\  0)^{T}$).
The belt~$0$ also could be an $l$-dimensional face, then  $\mathbf{w}_{0}=\mathbf{v}_{0}=(k_l^0\ k_{l}^{1}\ \ldots \ k_l^{l-1} \ k_l^l \ 0 \ldots \ 0)^{T}$.

\medskip
Let $v$ denote the volume of a cell of a mosaic. Then \RR\  $=\lim\limits_{i\to \infty}{\frac{v\cdot{b_{i+1}^d}}  {v\cdot{b_i^d}}}$ = \bb\ 
and \RS\ $=\lim\limits_{i\to \infty } {\frac{v\cdot b_i^d}{v\cdot\sum\limits_{j=0}^ib_j^d}}$ $=\lim\limits_{i\to \infty}\frac{b_i^d}{\sum\limits_{j=0}^ib_j^d}$ $=\lim\limits_{i\to \infty}\frac{b_i^d}{w_i^d}$ ($i \ge 1$). Thus we can calculate the limits by considering the numbers of the cells instead of the volumes.

\section{Belts with characteristic simplices}

We create the matrix $\mathbf{K}$ with elements $k_x^y$ in  the following way:
\begin{eqnarray*}
\mathbf{K}=
 \begin{pmatrix}
  k_0^0 & k_0^1 & \dots & k_0^d\\
  k_1^0 & k_1^1 & \dots & k_1^d\\
  \hdotsfor{4}\\
  k_d^0 & k_d^1 & \dots & k_d^d
 \end{pmatrix}.
\end{eqnarray*}
The definition of $k_x^y$ enables us to give the matrix $\mathbf{K}$ in another way
\begin{eqnarray}\label{eq:K}
\mathbf{K}
&=&
(\textrm{the number of  $y$-faces connecting to an $x$-face})\nonumber\\
&=&
 \left( \begin{smallmatrix}
  1 & N_0\{n_2,\dots ,n_d\} & N_1\{n_2,\dots ,n_d\} & \dots & N_{d-1}\{n_2,\dots ,n_d\}\\
  2 & 1 & N_0\{n_3,\dots ,n_d\} & \dots & N_{d-2}\{n_3,\dots ,n_d\}\\
  N_0\{n_1\} & N_1\{n_1\}& 1& \cdots & N_{d-3}\{n_4,\dots ,n_d\}\\
  N_0\{n_1,n_2\}& N_1\{n_1,n_2\}& N_2\{n_1,n_2\}&  \cdots & N_{d-4}\{n_5,\dots ,n_d\}\\
  \hdots & \hdots &\hdots &\hdots &\hdots \\
 N_0\{n_1,n_2,\dots ,n_{d-2}\} & N_1\{n_1,n_2,\dots ,n_{d-2}\} &
            N_2\{n_1,n_2,\dots ,n_{d-2}\} & \dots & 2 \\
  N_0\{n_1,n_2,\dots ,n_{d-1}\} & N_1\{n_1,n_2,\dots ,n_{d-1}\} &
            N_2\{n_1,n_2,\dots ,n_{d-1}\} & \dots & 1
 \end{smallmatrix}\right),
\end{eqnarray}
\noindent where $N_{l}\{m_1,m_2,\dots ,m_j\}$ denotes the number of  $l$-dimensional faces of the regular polyhedron
$\{m_1,m_2,\dots ,m_j\}$.

The $k_x^y$ also give  the number of  $y$-dimensional faces incident to an $x$-dimensional face. 
%{ So we can name the matrix $\mathbf{K}$ the \emph{incidencia matrix}}.

%\bigskip{\textbf{??? Remark.}}
%It gives  in case of the Euclidean mosaics $\{4,3,\dots,3,4\}$ that
%\begin{eqnarray}
%\mathbf{K}=(k_x^y)=\left\{\begin{aligned} 0 \ \ ,&\ \text{if } x<y ,\\
%                   \binom{y}{x}  ,& \text{ otherwise.} %\end{aligned}\right.
%\end{eqnarray}

\medskip

For a row $l$ of $\mathbf{K}$, we have the next lemma.
\begin{lemma}\label{lemma:euler} Properties of the matrix  $\mathbf{K}$, $(l\leq d)$.
\begin{enumerate}
\item [i.] $\sum\limits^{l}_{j=0}  {(-1)}^{j}k_l^j=1$.
\item [ii.] $\sum\limits^{d}_{j=l}  {(-1)}^{j}k_l^j=(-1)^d$.
\item [iii.] $\sum\limits^d_{j=l}  {(-1)}^{d-j}k_l^j=1$.
\item [iv.] $\sum\limits^d_{j=0}  {(-1)}^{j}k_l^j=1-(-1)^l+ (-1)^{d}$.
\end{enumerate}
\end{lemma}

\begin{proof}
We use the Euler-theorem for an $n$-dimensional polyhedron, which is
$ \sum\limits^{n-1} _{j=0}(-1)^jN_j = 1-(-1)^n,$ where ${N_j}$ is the number of $j$-dimensional faces of a polyhedron.
\begin{enumerate}
\item [\emph{i.}] The first $l$ elements of the row $l$ are the numbers of the facets of an $l$-dimensional polyhedron. So
\begin{eqnarray*}
\sum\limits^{l}_{j=0}  {(-1)}^{j}k_l^j&=& k_l^0- k_l^{1}+\cdots +(-1)^{j}k_l^{j}+ \cdots +(-1)^{l}k_l^{l}\\
        &=& N_{0} - N_{1}+ \cdots +(-1)^{j}N_{j}+ \cdots +(-1)^{(l-1)}N_{l-1} +(-1)^{l}\cdot 1 \\
        &=& 1-(-1)^{l} +(-1)^{l} =1.
\end{eqnarray*}

\item [\emph{ii.}] The last $l$ elements of the row $l$ are the numbers of the facets of an $(d-l)$-dimensional polyhedron. So
\begin{multline*}
 \sum\limits^{d}_{j=l}  {(-1)}^{j}k_l^j = {(-1)}^l k_l^l+ {(-1)}^{l+1} k_l^{l+1}+\cdots +(-1)^{j}k_l^{j}+ \cdots\\ +(-1)^{d}k_l^{d}
        = (-1)^l + {(-1)}^{l+1} N_{0} +{(-1)}^{l+2}N_{1}+ \cdots\\ +(-1)^{j}N_{(j-l)-1}+ \cdots +(-1)^{d}N_{(d-l)-1}
        = (-1)^{l+1} \Big(-1+  N_{0} -N_{1}+ \cdots\\ +(-1)^{j-(l+1)}N_{(j-l)-1}+ \cdots +(-1)^{d-(l+1)}N_{(d-l)-1} \Big) \\
        = (-1)^{l+1} \Big(-1+(1-(-1)^{d-l})\Big)\\= -(-1)^{d-l+l+1}=(-1)^{d+2}=(-1)^{d}.
 \end{multline*}
\item [\emph{iii.}]
\begin{eqnarray*}
\sum\limits^d_{j=l}  {(-1)}^{d-j}k_l^j &=& (-1)^{d}\sum\limits^d_{j=l} {(-1)}^{-j}k_l^j = (-1)^{d}\sum\limits^d_{j=l}  {(-1)}^{j}k_l^j \\
                &=& (-1)^{d}\cdot (-1)^{d} = 1.
\end{eqnarray*}
\item [\emph{iv.}]
\begin{eqnarray*}
\sum\limits^d_{j=0}  {(-1)}^{j}k_l^j &=& \sum\limits^{l}_{j=0}  {(-1)}^{j}k_l^j - (-1)^l k_l^l+  \sum\limits^{d}_{j=l}  {(-1)}^{j}k_l^j\\
                &=& 1-(-1)^l+ (-1)^{d}.\ \ 
\end{eqnarray*}
\end{enumerate}
\end{proof}

\begin{lemma} A $K_x$ consists of $k_x^dk_d^{d-1}\cdots k_2^1k_1^0$ pieces of characteristic simplices. A cell has $k_d^{d-1}\cdots k_2^1k_1^0$ characteristic simplices.
\end{lemma}

\begin{proof}  
There are as many characteristic simplices connected to a $p$-point as many ways are going from the $p$-point to a $0$-point through chains of the form: $p$-point, $(p-1)$-point, 
$(p-2)$-point, $\ldots$, $1$-point, $0$-point. The number of these chains is 
$k_d^{d-1}k_{d-1}^{d-2}\cdots k_2^1k_1^0$ and the number of $p$-points is $k_x^d$.  
\end{proof}

\noindent \textbf{Remark.} There also exist matrices $\mathbf{K}$ for regular mosaics in spherical spaces. In the spherical case, a regular mosaic gives rise to a regular polyhedron.

\section{Belts with cells}

In this section we consider a cell as a set of characteristic simplices having a common $d$-point. 
We take an arbitrary $x$-point and calculate the number of  $y$-points of the set of cells having
the common $x$-point. Based on the method of the logic sieve, 
we can get the value $g_x^y$ from the alternating sum of the products of the elements of row $x$ and column $y$ of matrix $\mathbf{K}$.

\begin{lemma}\label{lemma:gxy}
 $g_x^y =\sum\limits^d_{{\genfrac{}{}{0pt}{}{j=0,}{j\geq x,  j\geq y}}} {(-1)}^{d-j}k_x^j\cdot k_j^y$.
\end{lemma}
\begin{proof}
 The set of the cells $K_{d,j}$ $(j\in \{1, 2, \dots , n=k_x^d\})$  having a common $x$-point is the set $G_x$. It is the union of the
characteristic simplices having common $d$-points which are in the set of simplices having the common $x$-point.
In the cases \{4,3,4\} and \{4,3,5\},
Figure~\ref{abra:szimplex_pqr} shows the set whose elements are the cells having the common $1$-point and
it is the union of the characteristic simplices having common $3$-points, where the $3$-points are the vertices of
the characteristic simplices having the common $1$-point. 

\noindent Now we get, that
\begin{eqnarray*}
G_x^y = {K}_{d,1}^y\cup {K}_{d,2}^y\cup \cdots \cup {K}_{d,n}^y
        = \bigcup\limits_{j=1}^n {K}_{d,j}^y \label{eq:Gxy2Kdy}
\end{eqnarray*}
where $n=k_x^d$.
\begin{figure}[!htb]
\centering   
\begin{eqnarray*}
 \begin{pmatrix}
   &  &   &   &  &  \\
   &  & ({k_x^y}) & \cdots & k_x^{d-1} & k_x^d\\
   &  & \vdots &   &   &  \\
   &  & k_{d-1}^y&   &   &  \\
   &  & k_d^y & &     &  
 \end{pmatrix}
\end{eqnarray*} 
\caption{The elements of $\mathbf{K}$ whose alternating sum gives $g_x^y$ in Lemma \ref{lemma:gxy}.}\label{abra:maztix}
\end{figure}

In the following we determine the number $g_x^y = |G_x^y|$. The product $k_x^d k_d^y$ is larger than $g_x^y$, because the multiplicities of some $y$-points are larger than one.
Take the expression
\begin{eqnarray}\label{kxdkdy}
k_x^d k_d^y - k_x^{d-1}k_{d-1}^y + k_x^{d-2}k_{d-2}^y \cdots (-1)^{d-j}k_x^j k_j^y\cdots (-1)^{d-m}k_x^m k_m^y,
\end{eqnarray}
where $m$ is equal to the maximum of the set $\{x,y\}$, $d\geq j\geq l$, and 
we prove  that \eqref{kxdkdy} gives  $g_x^y$ without multiplicity. 

\noindent The product $k_x^j k_j^y$ is the number of  $y$-points connecting to the 
$j$-dimensional faces. Now we examine all the subexpressions $k_x^j k_j^y$ of \eqref{kxdkdy} considering the multiplicity of an arbitrary $y$-point. 
Let $l$ ($x\leq l\leq m$) be the minimum dimension, so that  $y$-point is on an 
$l$-dimensional common face with the vertex $x$, but not on any smaller-dimensional 
common faces. Then there is an $l$-point on the common $l$-face  with $y\in K_l$. 
The subexpression $k_x^d k_d^y$ of \eqref{kxdkdy} gives the number of the 
considered $y$-points as many times as the number of  $K_d$ around 
 $l$-point, so $k_l^d$ times. Similarly, in case of $k_x^{d-1}k_{d-1}^y$ the 
multiplicity of  $y$-point is the number of  $K_{d-1}$ around  $l$-point, so it is $k_l^{d-1}$. And similarly for the other 
terms. Using Lemma \ref{lemma:euler} we get
\begin{eqnarray*}
k_l^d - k_{l}^{d-1} + k_{l}^{d-2}- \cdots +(-1)^{d-l}k_l^l = \sum\limits^d_{j=l}  {(-1)}^{d-j}k_l^j=1.
\end{eqnarray*}
Thus the expression \eqref{kxdkdy} calculates the exact numbers (without multiplicity) of  $y$-points of  $G_x^y$. 
\end{proof}

We can write  $g_x^y$ in a matrix in the following way:
\begin{eqnarray*}
\mathbf{G}=
 \begin{pmatrix}
  g_0^0 & g_0^1 & \dots & g_0^d\\
  g_1^0 & g_1^1 & \dots & g_1^d\\
  \hdotsfor{4}\\
  g_d^0 & g_d^1 & \dots & g_d^d
 \end{pmatrix}.
\end{eqnarray*}

For a row $x$ of $\mathbf{G}$, we have the next lemma.
\begin{lemma}\label{lemma:sumgxj}
 $\sum\limits^d_{j=0}  {(-1)}^{j}g_x^j=1$.
\end{lemma}
\begin{proof}
\begin{eqnarray*}
\sum\limits^d_{j=0}  {(-1)}^{j}g_x^j &=& \sum\limits^d_{j=0}   {(-1)}^{j}
             \sum\limits^d_{{\genfrac{}{}{0pt}{}{i=0,}{i\geq x,  i\geq j}}}  {(-1)}^{d-i}k_x^i k_i^j =
       \sum\limits^d_{i=x} {(-1)}^{d-i}k_x^i
              \sum\limits^i_{j=0}  {(-1)}^{j}k_i^j    \\
                &=& \sum\limits^d_{i= x}  {(-1)}^{d-i}k_x^i  \cdot 1
                  =  \sum\limits^d_{i=x}  {(-1)}^{d-i}k_x^i=1.
\end{eqnarray*}
\end{proof}

\begin{lemma}\label{lemma:sumpartgxj}
 Let $H$ be the common part of a $G_x$ and a convex part of the mosaic and let the number 
of  $j$-points $(j\in \{0,1,\ldots,d\})$ of $H$ be $h_x^0$, $h_x^1$,\ldots, 
$h_x^d$, respectively. Then
 \begin{equation*}\label{eq:sumpartgxj}    
 \sum\limits^d_{j=0}  {(-1)}^{j}h_x^j=1.
 \end{equation*}
\end{lemma}

\begin{proof}
If $H$ is a $G_m$ $(0\leq m \leq x)$ then $h_x^j=g_m^j$ and Lemma \ref{lemma:sumgxj} applies when $x=m$.

\noindent In the following we examine the other cases.
As  $G_x$ is convex,  $H$ is also convex. If $H$ contains the two vertices of an edge, then it contains this edge too. If $H$ contains $i$ pieces of $i$-dimensional faces connecting to a common $(i+1)$-dimensional face, then it contains the common $(i+1)$-dimensional face as well $(i\in {0,1,\ldots,d-1})$. So $H$ always consists of whole faces of the cells. Let $l$ be the largest integer, for which $h_x^l>0$, the largest dimension of the faces of $H$ is $l$. Then $H$ consists of $h_x^l$ pieces of $l$-faces of the mosaic.

\noindent If $h_x^l=1$ then $H$ is an $l$-dimensional face (an $l$-dimensional  polyhedron). Then using the Euler-theorem for $H$ we get
\begin{equation}\label{eq:hx1}
h_x^0 - h_x^1 + \cdots +(-1)^{l-1}h_x^{l-1}+(-1)^lh_x^l = 1-(-1)^l+(-1)^l=1.
\end{equation}

\noindent If $h_x^l>1$ then we build up $H$ with $h_x^l$ pieces of $l$-faces. The 
alternating sum of any \mbox{$m$-dimensional} face (which is an $m$-dimensional regular 
polyhedron) is 1 from the equation \eqref{eq:hx1} when $l=m$. Let $H_1$ be an $l$-face. 
If we join two $l$-faces together along a common $(l-1)$-dimensional face, 
let it be $H_2$. Then from the sum of alternating sums of the two $l$-faces we have to 
extract the alternating sum of  $(l-1)$-face. So we get $(1+1)-1=1$, too. Join again 
to it another $l$-face, and we get $H_3$, and so on until we get $H$. 
Generally, let $m$ be the largest dimensional common face of $H_i$ and the next $l$-face. 
In this case the number of the common $m$-faces is 1. The alternating sum of $H_{i+1}$ 
is $2 - (k_m^0 +\cdots+ (-1)^{m}k_m^m) = 1$, where $k_m^y$ gives  how many common 
$y$-faces there are. The value $k_m^y$ also shows how many $y$-faces connect to an $m$-face on the $l$-face. So 
it is an element of the matrix $\mathbf{K}$ of the regular $l$-face and the alternating 
sum $k_m^0 +\cdots+ (-1)^{m}k_m^m$ is 1 from Lemma \ref{lemma:euler}. \end{proof}

\begin{lemma}\label{lemma:rankG}
\emph{Rank}$(\mathbf{G})=d+1$.
\end{lemma}
\begin{proof} Take the column vectors of $\mathbf{G}$. They are independent if and only if the only linear combination which gives the vector zero is the linear combination with all coefficients zero $\alpha_j\ (j\in {0,1,\ldots,d})$.  In case of any row~$x$ (the $x$th coordinates of the vectors) we get 
\begin{eqnarray*}
\sum\limits^d_{j=0}  {\alpha_j}g_x^j &=& \sum\limits^d_{j=0}   {\alpha_j}
             \sum\limits^d_{{\genfrac{}{}{0pt}{}{i=0,}{i\geq x,  i\geq j}}}  {(-1)}^{d-i}k_x^i k_i^j =
       \sum\limits^d_{i=x} 
              \sum\limits^i_{j=0}  {\alpha_j} {(-1)}^{d-i}k_x^i k_i^j  \\
                &=& \sum\limits^d_{i=x} {(-1)}^{d-i}k_x^i
                              \sum\limits^i_{j=0}  {\alpha_j}k_i^j   =0.\ 
\end{eqnarray*}
It is equal to zero if and only if the coefficients $\alpha_j$ are the only 
solutions of the linear equation system 
$\sum\limits^i_{j=0}  {\alpha_j}k_i^j=0$, $(i\in {x,\ldots,d})$. For row~$x=0$, we can 
write the equation system in the form below
\begin{eqnarray*}
0 &=&k_0^0 \alpha_0 \\
0 &=& k_1^0 \alpha_0 + k_1^1\alpha_1 \\
0 &=& k_2^0 \alpha_0 + k_2^1 \alpha_1 + k_2^2 \alpha_2\\
&\cdots&  \\
0 &=& k_d^0 \alpha_0 + k_d^1 \alpha_1 + \cdots + k_d^d \alpha_d,
\end{eqnarray*}
which results that all  $\alpha_j$  must be $0$, because none of $k_x^y$ are zero.  
\end{proof}

\begin{lemma}\label{lemma:wxy}
 \begin{equation} \label{eq:wxy} 
  w_{i+1}^y =\sum\limits^d_{j=0} {(-1)}^{j}w_i^{j}\cdot g_{j}^y.
 \end{equation}
\end{lemma}
\begin{proof} The cells of the belt $(i+1)$ come from among the cells of the mosaic having common vertices with the belt $i$ but not in $W_i$. And $W_{i+1}=B_{i+1}\cup W_i=\bigcup\limits_{j\in W_i^0}
G_{0,j} $.

\noindent We take an arbitrary $y$-point in the belt $(i+1)$. 

\noindent When it is in the belt $i$, then it is contained by $g_y^0$ pieces of $G_0$ and 
its multiplicity is $g_y^0$ calculated by subexpression $w_i^{0}\cdot g_{0}^y$. In a 
similar way we get that the multiplicity of  $y$-point is $g_y^j$ using the 
expression $w_i^{j}\cdot g_{j}^y$. Now we sum the multiplicity with alternating signs based on \eqref{eq:wxy} and after using Lemma \ref{lemma:sumgxj} we get 
that the multiplicity of any $y$-point is exactly one.  

\noindent If  $y$-point is on the belt $i$ or in the belt $(i+1)$ we can prove in a similar way applying Lemma \ref{lemma:sumpartgxj} that the multiplicity of  $y$-point is also one. 
\end{proof}

\bigskip

Let define
\begin{eqnarray}\label{eq:defM}
\mathbf{M}=\mathbf{G}^T
 \begin{pmatrix}
  1 & 0  & 0 & \dots & 0\\
  0 & -1 & 0 & \dots & 0\\
  0 & 0  & 1 & \dots & 0\\
  \hdotsfor{5}\\
  0 & 0  & 0 & \dots & (-1)^d
 \end{pmatrix}.
\end{eqnarray}
Then using Lemma \ref{lemma:wxy} for all $y$ we get the form
\begin{eqnarray}\label{eq:wMw}
\mathbf{w}_{i+1}=\mathbf{M}\mathbf{w}_i
\end{eqnarray}
and furthermore
\begin{eqnarray*}\label{bMw}
\mathbf{v}_{i+1}=\mathbf{w}_{i+1}-\mathbf{w}_{i}=\mathbf{M}\mathbf{w}_i-\mathbf{I}\mathbf{w}_i=(\mathbf{M-I})\mathbf{w}_i,
\end{eqnarray*}
where $\mathbf{I}$ is the identity matrix. 

Lemma \ref{lemma:rankG} provides that rank($\mathbf{G})=d+1$ and looking at \eqref{eq:defM} we have the following result: 
\begin{lemma}\label{lemma:rankM}
\emph{Rank}$(\mathbf{M})=d+1$.
\end{lemma}

\section{Eigenvalues of matrix M}\label{sec:lemmas}

We get the following linear recursion for the sequences $w_i^y$  $(y \in {0,1,2,\ldots,d}$) in a matrix form, where the index of recursion is $i$ $(i\geq 0)$: $\mathbf{w}_{i+1}=\mathbf{M}\mathbf{w}_i$. The recursive sequences $w_i^y$ are defined by \eqref{eq:wxy}. Let $n=d+1$ be the rank$(\mathbf{M})$.

Let the sequence $\left\{ r_i\right\}_{i=1}^\infty$  be defined by
\begin{eqnarray}
r_{i}&=&{\balpha} ^T\mathbf{w}_i, \label{eq:ra}
\end{eqnarray}
\noindent where $\mathbf{\balpha}$ is a real vector. The coordinates of $\mathbf{v}_i$ $(i\geq 1)$ also satisfy the equation \eqref{eq:ra}, since
$\mathbf{v}_{i}=(\mathbf{M-I})\mathbf{M}^{-1}\mathbf{w}_i$ and $\mathbf{e}_y^T\mathbf{v}_{i}=  b_{i}^y=\mathbf{e}_y^T(\mathbf{I}-\mathbf{M}^{-1})\mathbf{w}_i$, so $\balpha^T=\mathbf{e}_y^T(\mathbf{I}-\mathbf{M}^{-1})$, where $\mathbf{e}_y^T$ is the $y$th normal basis vector. 

Let 
\begin{eqnarray}
z^n={\beta}_1z^{n-1}+{\beta}_2z^{n-2}+ \dots +{\beta}_nz^{0} \label{eq:karartrisztikus egyenlet}
\end{eqnarray}
be the characteristic equation of the matrix $\mathbf{M}$, where $\beta_j \in {\mathbb R}$ and $\beta_n\neq0$ for rank$(\mathbf{M})=n$.

\begin{lemma}
The characteristic equation of the recursive sequence $r_i$ and matrix $\mathbf{M}$ are the same.  
\end{lemma} 

\begin{proof}  Using the theorem of {Cayley--Hamilton}, $z$ can be 
substituted by $\mathbf{M}$ in \eqref{eq:karartrisztikus egyenlet}, and we get
\begin{multline}\label{eq:Ma3}
\mathbf{M}^{n}=  \ {\beta}_1 \mathbf{M}^{n-1}+\dots + {\beta}_j\mathbf{M}^{n-j}+ \dots+{\beta}_n \mathbf{M}^{0}\\
\balpha ^T \mathbf{M}^{n}\mathbf{w}_{i-n} = \
\balpha ^T\left({\beta}_1 \mathbf{M}^{n-1}+\dots + {\beta}_j\mathbf{M}^{n-j} + 
\dots+{\beta}_n \mathbf{M}^{0}\right)\mathbf{w}_{i-n}\\
 = \ {\beta}_1 \balpha ^T\mathbf{M}^{n-1}\mathbf{w}_{i-n}+ \dots + {\beta}_j
\balpha ^T\mathbf{M}^{n-j}\mathbf{w}_{i-n}+\dots+
{\beta}_n \balpha ^T\mathbf{M}^{0}\mathbf{w}_{i-n}.
\end{multline}

\noindent In case of $1\leq i, j$ from \eqref{eq:wMw} and \eqref{eq:ra} we get 
\begin{eqnarray}
 r_{j}&=&\balpha ^T\mathbf{M}^{j-i}\mathbf{w}_{i}, \nonumber\\
  \noalign{\noindent and}
     r_{i-j}&=&\balpha ^T\mathbf{M}^{n-j}\mathbf{w}_{i-n}, \quad\quad (i\geq j+1,\  i\geq n+1).\label{eq:rm3}
\end{eqnarray}

\noindent Substituting  \eqref{eq:rm3} into \eqref{eq:Ma3} we get  
\begin{eqnarray*}
  r_i &=& {\beta}_1r_{i-1}+ \dots +{\beta}_jr_{i-j} + \dots  +{\beta}_nr_{i-n}, \quad\quad (i\geq n+1).
\end{eqnarray*}

\noindent So,  $r_i$ is a linear recursive sequence with (at most) rank $n$ and then its characteristic equation is also \eqref{eq:karartrisztikus egyenlet}.  
\end{proof}

\medskip

The factorization of the common characteristic polynomial is:
\begin{eqnarray*}
z^n-{\beta}_1z^{n-1}-{\beta}_2z^{n-2}- \dots -{\beta}_nz^{0}=(z-z_1)^{m_1}\cdots(z-z_h)^{m_h},
\label{eq:karartrisztikus egyenlet_gyokok}
\end{eqnarray*}
\noindent where $z_1$, $\dots$, $z_h$ are non zero, different roots and $m_1+\cdots+m_h=n,\  1\leq h\leq n$.

Any elements of the linear recursive sequence $r_i$ can be determined explicitly because of the theorem of recursive sequences {(\cite[p.33]{sho})} as follows
\begin{eqnarray*}
r_i=g_1(i)z_1^i+g_2(i)z_2^i+\dots +g_h(i)z_h^i, \label{eq:gyokokkel}
\end{eqnarray*}
\noindent where $g_k(i)$ are the polynomials in $i$ with degree at most
$(m_k-1)$ and depend on $r_1$, $r_2$, $\dots$, $r_n$,  $m_k$ and  $z_k$ $(k=1,\dots, h)$. If $z_k$ is a simple root, $m_k=1$, then $g_k(i)=g_k$ is a constant. 

\bigskip
Now we assume that all the roots of the characteristic equation \eqref{eq:karartrisztikus egyenlet} are real ($z_k \in {\mathbb R}$, $k=1,\dots, h\leq  n$)  and  $r_i\neq0$ $(i\geq1)$. And let $s_i=\sum_{j=0}^ir_j$.   

\begin{lemma}\label{tetel:lim_rr}
Let  $1< h\leq n$, $\left|z_1\right|>\left|z_k\right|\neq 0 $, $\left|z_1\right|>1$, ${g_1\neq0}$
$(k=2,\dots, h)$, then  $\lim\limits_{ i\to \infty }{\frac{r_{i+1}}{r_i}}=z_1$ and
$\lim\limits_{ i\to \infty} {\frac{r_i}{s_i}}= {\frac{z_1-1}{z_1}}$ $(i\geq1)$.
\end{lemma}

\begin{proof}
In the case  $h>1$, $\left|z_1\right|>\left|z_k\right|$,  $\left|z_1\right|>1$, $g_1\neq0$ $(k=2,\dots, h)$.\\
As $\lim\limits_{ i\to \infty } \left({\frac{z_k}{z_1}}\right)^i =0$ and
  $\lim\limits_{ i\to \infty } {\frac{z_k^j}{z_1^i}} =0$ \ \  $(2\leq k\leq h,\  j\leq i)$, then
\begin{eqnarray*}
 \lim\limits_{ i\to \infty }{\frac{r_{i+1}}{r_i}}&=&
 \lim\limits_{ i\to\infty}{\frac{g_1z_1^{i+1}+g_2(i+1)z_2^{i+1}+\dots +g_h(i+1)z_h^{i+1}}
         {g_1z_1^i+g_2(i)z_2^i+\dots +g_h(i)z_h^i}}\\
 &=&\lim\limits_{ i\to\infty}
 {\frac{g_1z_1+g_2(i+1)z_2\left(\frac{z_2}{z_1}\right)^i +\dots + g_h(i+1)z_h\left(\frac{z_h}{z_1}\right)^i}
    {g_1+g_2(i)\left(\frac{z_2}{z_1}\right)^i +\dots + g_h(i)\left(\frac{z_h}{z_1}\right)^i}}= z_1.
\end{eqnarray*}
Furthermore
\begin{eqnarray*}
\lim\limits_{ i\to \infty }{\frac{r_i}{s_i}}&=&
 \lim\limits_{ i\to \infty }{\frac{r_i}{\sum\limits_{j=0}^ir_j}}=
 \lim\limits_{ i\to \infty }{\frac{g_1z_1^i+g_2(i)z_2^i+\dots +g_h(i)z_h^i}
 {g_1\sum\limits_{j=0}^iz_1^j + \sum\limits_{j=0}^ig_2(j)z_2^j +\dots+\sum\limits_{j=0}^ig_h(j)z_h^j}}\\
&=&\lim\limits_{ i\to \infty }{\frac{g_1z_1^i+g_2(i)z_2^i+\dots +g_h(i)z_h^i}{g_1\frac{z_1^{i+1}-1}{z_1-1}+
\sum\limits_{j=0}^ig_2(j)z_2^j +\dots+\sum\limits_{j=0}^ig_h(j)z_h^j}}\\
 &=&\lim\limits_{i \to \infty }\frac{g_1+g_2(i)\left(\frac{z_2}{z_1}\right)^i +\dots +
g_h(i)\left(\frac{z_h}{z_1}\right)^i} {g_1\frac{z_1-\frac{1}{z_1^i}}{z_1-1}+
\sum\limits_{j=0}^ig_2(j)\frac{z_2^j}{z_1^i} +\dots+\sum\limits_{j=0}^ig_h(j)\frac{z_h^j}{z_1^i}}=
\frac{z_1-1}{z_1}.
\end{eqnarray*}
\end{proof}

\begin{lemma}\label{tetel:lim_ss} If
$1< h\leq n$, $\left|z_1\right|>\left|z_k\right|\neq0$, $\left|z_1\right|>1$, $g_1\neq0$ $(k=2,\dots, h)$, then $\lim\limits_{i\to \infty }{\frac{s_{i+1}}{s_i}}=\lim\limits_{ i\to \infty }{\frac{r_{i+1}}{r_i}}$ $(i\geq1)$.
\end{lemma}
\begin{proof} 
It is similar to the previous cases.

\begin{eqnarray*}
 &&\lim\limits_{ i\to \infty }{\frac{s_{i+1}}{s_i}}\ = \ 
 \lim\limits_{ i\to \infty }\frac{\sum\limits_{j=0}^{i+1}{r_j}}{\sum\limits_{j=0}^{i}r_j}=
 \lim\limits_{ i\to \infty }
 \frac{g_1\sum\limits_{j=0}^{i+1}z_1^j + \sum\limits_{j=0}^{i+1}g_2(j)z_2^j +\dots+\sum\limits_{j=0}^{i+1}g_h(j)z_h^j}
 {g_1\sum\limits_{j=0}^iz_1^j + \sum\limits_{j=0}^ig_2(j)z_2^j +\dots+\sum\limits_{j=0}^ig_h(j)z_h^j}
 \\
 &=&\!\!\!\lim\limits_{ i\to \infty }\!\!\!
 \frac { g_1z_1\frac{z_1-\frac{1}{z_1^i}}{z_1-1}\!+ g_2(i+1)z_2\left(\frac{z_2}{z_1}\right)^i\! \! + \!\!\sum\limits_{j=0}^{i}g_2(j)\frac{z_2^j}{z_1^i} +\!\dots\!+  g_h(i+1)z_h\left(\frac{z_h}{z_1}\right)^i\! \!+
\!\!\sum\limits_{j=0}^{i}g_h(j)\frac{z_h^j}{z_1^i}}
 { g_1\frac{z_1-\frac{1}{z_1^i}}{z_1-1}+
\sum\limits_{j=0}^ig_2(j)\frac{z_2^j}{z_1^i} +\dots+  \sum\limits_{j=0}^ig_h(j)\frac{z_h^j}{z_1^i}}\\
  &=&\!  {z_1}=\lim\limits_{i\to \infty }\frac{r_{i+1}}{r_{i}}.
\end{eqnarray*}
\end{proof}

\section{Proof of Main theorem}

The proof is the summarising of the previous sections. For the limits 
 \RR\  and \RS\ we have to calculate the numbers of the cells in the belts. 
First, we can construct a matrix $\mathbf{K}$ for every hyperbolic regular mosaic. 
Second, from matrix $\mathbf{K}$ we get the matrix $\mathbf{G}$, where an element $g_x^y$ is the alternating sum of the products of the elements of row $x$ and column $y$ of matrix $\mathbf{K}$. 
Third, this matrix $\mathbf{G}$ generates the matrix ${\mathbf{M}}$, which contains the coefficients of the system of the linear recurrence sequences given by the matrix form $\mathbf{w}_{i+1}=\mathbf{M}\mathbf{w}_i$. 
Recall, the $y\text{th}$  coordinate ($y \in {0,1,2,\ldots,d}$; defined by \eqref{eq:wxy}) of $\mathbf{w}_i$ equals to the number of  $y$-points in the union of belts $j$ $(0\leq j\leq i)$, 
moreover, the $y\text{th}$  coordinate of the vector $\mathbf{v}_{i}=(\mathbf{M-I})\mathbf{M}^{-1}\mathbf{w}_i$ gives the number of $y$-points in belts $i$.

Finally, let $r_i=b_i^d$, the number of cells in the  belt $i$. Recall, $b_{i}^y=\mathbf{e}_y^T(\mathbf{I}-\mathbf{M}^{-1})\mathbf{w}_i$. 
The lemmas of Section~\ref{sec:lemmas} provide, that for the limits we have to give the largest eigenvalue $z_1$ of matrix ${\mathbf{M}}$.  Thus \RR\ $=z_1$ and because of the algebra of limits \RrS\ $=\lim\limits_{i\to \infty }({\frac{V_{i+1}}{V_{i}}\cdot \frac{V_{i}}{S_{i}})} =$ \RR\ $\cdot$ \RS\ $=$ $z_1\cdot \frac{z_1-1}{z_1}= z_1-1$.

The exact values and decimal approximations of the limits for hyperbolic regular mosaics can be found at the end of Introduction. In the next subsection we give and enumerate the important results of the calculations of the limits. 

%$s_i=w_i^d$ $(i\geq1)$.

%Also for the spherical mosaics $\{3,3\}$, ......$\{..\}$ the coordinates of vectors $\mathbf{w}_i$ and $\mathbf{v}_i$ help to determine the numbers of the elements of belt $i$.   

\begin{remark}
Two mosaics are called dual mosaics if the orders of their Schl\"afli symbols are opposite to each other. For example, mosaics $\{4,3,5\}$ and $\{5,3,4\}$  are dual to each other. We can see that in case of dual mosaics the limits are the same (see: \cite{nem2}), because the vertices of the characteristic simplices are inverted,  $l$-points become $(d-l)$-points  $(0\leq l\leq d)$ of the dual mosaic. Thus, the number of the vertices in the belts are equal to the number of the cells in the belts  of the dual mosaic. The consequence is that the limits are the same.
\end{remark}

\begin{remark}
We shall also get the same limits if we define belt~$0$ with an arbitrary $l$-dimensional face $(0\leq l\leq d)$, then  \[\mathbf{w}_{0}=\mathbf{v}_{0}=(k_l^0\ k_{l}^{1}\ \ldots \ k_l^{l-1} \ k_l^l \ 0 \ldots \ 0)^{T}\] and if we define the sequence $r_i$, $(i\geq0)$ by the number of  any $l$-dimensional face $(0\leq l\leq d)$, so $r_i=b_i^l$ and  $s_i=w_i^l$.
\end{remark}

\subsection{Results for the mosaics}

First of all we give the matrices $\mathbf{K}$ and $\mathbf{M}$ for all the hyperbolic regular mosaics with bounded cells.

In 2-dimensional space the Schl\"afli symbols of the regular mosaics are $\{p,q\}$. If $\frac1p+\frac1q<\frac12$, they are in the 
hyperbolic plane, if  $\frac1p+\frac1q=\frac12$, they are the
Euclidean regular mosaics  $\{3,6\}$, $\{6,3\}$ and $\{4,4\}$.
The matrices $\mathbf{K}$ and $\mathbf{M}$ can always be written in the form
\begin{eqnarray*}
\mathbf{K}=
\begin{pmatrix}
1 & q & q \\
2 & 1 & 2 \\
p & p & 1 \\
\end{pmatrix},\quad
\mathbf{M}=
\begin{pmatrix}
pq-2q+1 & -2p+2  & p \\
pq-q    & -2p+1  & p \\
q       & -2     & 1 \\
\end{pmatrix}.
\end{eqnarray*}

In 3-dimensional Euclidean and hyperbolic spaces we can also obtain the matrix $\mathbf{M}$ 
in another way. The order of the transformation group that fixes the 0-point is
$V=\frac{4\pi}{\frac{\pi}{r}+\frac{\pi}{q}-\frac{\pi}{2}}=\frac{4}{\frac1r+\frac1q-\frac12}=\frac{8rq}{4-(r-2)(q-2)}$.
So the number of the simplices of $K_0$ is $V$. There is only one 0-point of $K_0$, so $k_0^0=1$. The number of characteristic simplices  having 
common 01-edge is $r$, so $k_0^1=\frac{V}{2r}$. Similarly,
$k_0^2=\frac{V}{4}$ and $k_0^3=\frac{V}{2q}$.
The number of the elements of $K_1$, $K_2$ and $K_3$ are $4r$, $4p$ and
$U=\frac{4\pi}{\frac{\pi}{p}+\frac{\pi}{q}-\frac{\pi}{2}}=\frac{4}{\frac1p+\frac1q-\frac12}=\frac{8pq}{4-(p-2)(q-2)}$,
respectively.

If $x = y$, then $k_x^y=1$.  Otherwise $k_x^y$ is equal to the ratio of 
$|K_x|$ and the number of simplices having common $xy$-edge.

Summarising we get the matrix $\mathbf{K}$ (which is the same as \eqref{eq:K}); 
\begin{eqnarray*}
\mathbf{K}=
\begin{pmatrix}
1 & \frac{V}{2r} & \frac{V}{4} & \frac{V}{2q}\\
2 & 1 & r & r \\
p & p & 1 & 2 \\
\frac{U}{2q} & \frac{U}{4} & \frac{U}{2p} & 1 \\
\end{pmatrix}.
\end{eqnarray*}

Details for regular mosaics in 3-dimensional and 4-dimensional spaces are presented below. 

\noindent Mosaic $\{4,3,5\}$:\\
\phantom{x}\hspace{0.4cm}
$\mathbf{K}=\begin{pmatrix}
1 &  12  & 30 & 20 \\
2 & 1& 5 &  5 \\
4 & 4 & 1 &  2 \\
8 &  12 & 6 & 1 \\
\end{pmatrix}$,\ 
$\mathbf{M}=\begin{pmatrix}
63 &  -22 & 12 & -8 \\
132 & -41 & 20 & -12 \\
90 & -25 & 11 & -6 \\
20 &  -5 & 2 & -1 \\
\end{pmatrix}$.
\medskip

\noindent Mosaic $\{5,3,4\}$:\\
\phantom{x}\hspace{0.4cm}
$\mathbf{K}=\begin{pmatrix}
1 &  6  & 12 & 8 \\
2 & 1& 4 & 4  \\
5 & 5  &  1 &  2 \\
20 & 30   & 12  & 1 \\
\end{pmatrix}$,
$\mathbf{M}=\begin{pmatrix}
111 &  -62 & 35 & -20 \\
186 & -101 & 55 & -30 \\
84 & -44 & 23 & -12 \\
8 &  -4 & 2 & -1 \\
\end{pmatrix}$.
\medskip

\noindent Mosaic $\{3,5,3\}$:\\
\phantom{x}\hspace{0.4cm}
$\mathbf{K}=\begin{pmatrix}
1 &  20 & 30 &  12\\
2 &  1  & 3  &  3 \\
3 &  3  &  1 &  2 \\
12 &  30 & 20 & 1 \\
\end{pmatrix}$,
$\mathbf{M}=\begin{pmatrix}
93 &  -29 & 21 & -12 \\
290 & -82 & 57 & -30 \\
210 & -57 & 39 & -20 \\
12 &  -3 & 2 & -1 \\
\end{pmatrix}$.
\medskip

\noindent Mosaic $\{5,3,5\}$:\\
\phantom{x}\hspace{0.4cm}
$\mathbf{K}=\begin{pmatrix}
1 & 12& 30& 20 \\
2 & 1 & 5 & 5 \\
5 & 5 & 1 & 2 \\
20& 30&12 & 1 \\
\end{pmatrix}$,
$\mathbf{M}=\begin{pmatrix}
	273 &  -77 & 35 & -20 \\
	462 & -126 & 55 & -30 \\
	210 & -55 & 23 & -12 \\
	20 &  -5 & 2 & -1 \\
\end{pmatrix}$.
\bigskip

\noindent Mosaic $\{4,3,3,5\}$:\\
\phantom{x}\hspace{0.4cm}$\mathbf{K}=\begin{pmatrix}
1 & 120 & 720 & 1200 & 600\\
2 & 1  &   12 & 30 & 20\\
4 & 4  &    1 & 5 & 5\\
8 & 12 &   6 & 1 & 2\\
16 & 32 &  24 & 8 & 1\\
\end{pmatrix}$,
$\mathbf{M}=\begin{pmatrix}
2641 &  -126 & 44 & -24 & 16\\
7560 & -327 & 104 & -52 & 32\\
7920 & -312 & 91 & -42 & 24\\
3600 &  -130 & 35 & -15 & 8\\
600 &   -20 &  5 &  -2 & 1\\
\end{pmatrix}$.

\medskip 
\noindent Mosaic $\{5,3,3,4\}$:\\
\phantom{x}\hspace{0.4cm} $\mathbf{K}=\begin{pmatrix}
1 & 8 & 24 & 32 & 16\\
2 & 1  &   6 & 12 & 8\\
5 & 5  &    1 & 4 & 4\\
20 & 30 &   12 & 1 & 2\\
600 & 1200 &  720 & 120 & 1\\
\end{pmatrix}$,
$\mathbf{M}=\begin{pmatrix}
9065 &  -4588 & 2325& -1180 & 600\\
18352 & -9269 & 4685 & -2370 & 1200\\
11160 & -5622 & 2833 & -1428 & 720\\
1888 &  -948 & 476 & -239 & 120\\
16 &   -8 &  4 &  -2 & 1\\
\end{pmatrix}$.

\medskip 
\noindent Mosaic $\{3,3,3,5\}$:\\
\phantom{x}\hspace{0.4cm} $\mathbf{K}=\begin{pmatrix}
1 & 120 & 720 & 1200 & 600\\
2 & 1  &   12 & 30 & 20\\
3 & 3  &    1 & 5 & 5\\
4 & 6 &   4 & 1 & 2\\
5 & 10 &  10 & 5 & 1\\
\end{pmatrix}$,
$\mathbf{M}=\begin{pmatrix}
121 &  -14 & 8 & -6 & 5\\
840 & -55 & 23 & -14 & 10\\
1920 & -92 & 31 & -16 & 10\\
1800 &  -70 & 20 & -9 & 5\\
600 &   -20 &  5 &  -2 & 1\\
\end{pmatrix}$.

\medskip
\noindent Mosaic $\{5,3,3,3\}$:\\
\phantom{x}\hspace{0.4cm} $\mathbf{K}=\begin{pmatrix}
1 & 5 & 10 & 10 & 5\\
2 & 1  &   4 & 6 & 4\\
5 & 5  &    1 & 3 & 3\\
20 & 30 &   12 & 1 & 2\\
600 & 1200 &  720 & 120 & 1\\
\end{pmatrix}$,
$\mathbf{M}=\begin{pmatrix}
2841 &  -2298 & 1745& -1180 & 600\\
5745 & -4639 & 3515 & -2370 & 1200\\
3490 & -2812 & 2125 & -1428 & 720\\
590 &  -474 & 357 & -239 & 120\\
5 &   -4 &  3 &  -2 & 1\\
\end{pmatrix}$.

\medskip
\noindent Mosaic $\{5,3,3,5\}$:\\
\phantom{x}\hspace{0.4cm} $\mathbf{K}=\begin{pmatrix}
1 & 120 & 720 & 1200 & 600\\
2 & 1  &   12 & 30 & 20\\
5 & 5  &    1 & 5 & 5\\
20 & 30 &   12 & 1 & 2\\
600 & 1200 &  720 & 120 & 1\\
\end{pmatrix}$,
$\mathbf{M}=\begin{pmatrix}
339361 &  -11458 & 2905 & -1180 & 600\\
687480 & -23159 & 5855 & -2370 & 1200\\
418320 & -14052 & 3541 & -1428 & 720\\
70800 &  -2370 & 595 & -239 & 120\\
600 &   -20 &  5 &  -2 & 1\\
\end{pmatrix}$. 

\medskip

Now we summarise  the eigenvalues and the (rounded) values of $g_1$ for all the mosaics for the cases where the belt~$0$ is  a cell (or a vertex). The values of $g_1$s are in brackets.  They all satisfy the conditions of  Lemmas \ref{tetel:lim_rr} and \ref{tetel:lim_ss}. The calculation was made using Maple software. 

\medskip

\noindent $\{4,3,5\}$: $15\pm 4\sqrt{14}$,  1, 1;  3.8571, (0.8304).

\noindent $\{5,3,4\}$: $15\pm 44\sqrt{14}$, 1, 1;  2.1429, (0.3322).

\noindent $\{5,3,5\}$: $\frac{167}{2}\pm \frac{13}{2}\sqrt{165}$,  1, 1;  1.6364, (0.1266).

\noindent $\{3,5,3\}$: $\frac{47}{2}\pm \frac{21}{2}\sqrt{5}$,  1, 1; 2, (0.2918).

\noindent $\{3,3,3,5\}$: $22\pm\sqrt{401}\pm 2\sqrt{221+11\sqrt{401}}$, 1;  117.6044, (8.8448).

\noindent $\{5,3,3,3\}$: $22\pm\sqrt{401}\pm 2\sqrt{221+11\sqrt{401}}$, 1;  1.4629, (0.0737).

\noindent $\{4,3,3,5\}$: $\frac12({1195\pm 11\sqrt{11641}\pm\sqrt{2836582+26290\sqrt{11641}}})$, 1;  3.8242, (0.2584).

\noindent $\{5,3,3,4\}$: $\frac12({1195\pm 11\sqrt{11641}\pm\sqrt{2836582+26290\sqrt{11641}}})$, 1;  1.1090, (0.0069).

\noindent $\{5,3,3,5\}$:  $79876\pm 3135\sqrt{649}\pm 2\sqrt{3189673350+125205630\sqrt{649}}$, 1;   1.0622, (0.0019). 

\medskip

\noindent $\{p,q\}$: $z_1=\frac{c+\sqrt{c^2-4}}{2}>1$, $1=\frac{c-(c-2)}{2}=\frac{c-\sqrt{(c-2)^2}}{2}> z_2=\frac{c-\sqrt{c^2-4}}{2}>0$, $z_3=1$,  where $c=(p-2)(q-2)-2>4$; Since $r_1=g_1z_1+g_2z_2+g_3$, $r_2=g_1z_1^2+g_2z_2^2+g_3$, $r_3=g_1z_1^3+g_2z_2^3+g_3$ and
$r_1-r_2+r_3-r_2=g_1z_1(1-z_1+z_1^2-z_1)+g_2z_2(1-z_2+z_2^2-z_2)=g_1z_1(1-z_1)^2+g_2z_2(1-z_2)^2>0$, then
 $g_1=\frac{z_2(r_1-r_2)+(r_3-r_2)}{z_1(z_1-z_2)(z_1-1)}>\frac{(r_1-r_2)+(r_3-r_2)}{z_1(z_1-z_2)(z_1-1)}>0.$

\end{document}